\documentclass[11pt]{amsart}
\usepackage[usenames]{color}
\usepackage{amsfonts}
\usepackage{amsmath}
\usepackage{amssymb}
\usepackage{setspace}
\usepackage[dvips]{graphicx}
\usepackage{verbatim}

\newtheorem{theorem}{Theorem}
\newtheorem{prop}[theorem]{Proposition}

\newtheorem{lemma}[theorem]{Lemma}

\newtheorem{example}[theorem]{Example}
\newtheorem{remark}[theorem]{Remark}

\newcommand{\e}{\varepsilon}

\newcommand{\D}{\mathcal{D}}
\newcommand{\I}{\mathcal{I}}

\newcommand{\RR}{\mathbb{R}}

\newcommand{\ZZ}{\mathbb{Z}}
\newcommand{\NN}{\mathbb{N}}

\newcommand{\udim}{\overline{\dim}}
\newcommand{\ldim}{\underline{\dim}}

\def\a{\alpha}
\def\b{\beta}
\def\g{\gamma}

\def\ux{\underline{x}}

\newcommand{\on}{\mathcal{OFF}_n}
\newcommand{\of}{\mathcal{OFF}}
\newcommand{\fr}{\mathrm{Fr}_a}

\title{On the dimension of iterated sumsets}

\author{J\"{o}rg Schmeling}
\address{J\"{o}rg Schmeling. Centre for Mathematical Sciences. Lund University. Lunds Institute of Technology, Po Box 118, SE--22100 Lund, Sweden}
\email{joerg@maths.lth.se}
\author{Pablo Shmerkin}
\address{Pablo Shmerkin, Centre for Interdisciplinary Computational and Dynamical Analysis and School of Mathematics, Alan Turing Building, University of Manchester, Oxford Road, Manchester M13 9PL, UK}
\email{Pablo.Shmerkin@manchester.ac.uk}

\thanks{Both authors were members of the Mathematical Sciences Research Institute when this project was started. We are grateful to MSRI for the fruitful atmosphere and financial support. P.S. acknowledges support from EPSRC grant EP/E050441/1 and the University of Manchester.}

\begin{document}

\begin{abstract}
Let $A$ be a subset of the real line. We study the fractal dimensions of the $k$-fold iterated sumsets $kA$, defined as
\[
k A = \{ a_1+ \ldots + a_k : a_i \in A\}.
\]
We show that for any non-decreasing sequence $\{ \alpha_k \}_{k=1}^\infty$ taking values in $[0,1]$, there exists a compact set $A$ such that $k A$ has Hausdorff dimension $\alpha_k$ for all $k\ge 1$. We also show how to control various kinds of dimension simultaneously for families of iterated sumsets.

These results are in stark contrast to the Pl\"{u}nnecke-Rusza inequalities in additive combinatorics. However, for lower box-counting dimension, the analogue of the Pl\"{u}nnecke-Rusza inequalities does hold.
\end{abstract}

\maketitle

\section{Introduction and statement of results}

Given sets $A, B$ in some ambient group, let $A+B = \{ a+b : a\in A, b\in B\}$ and $kA = \{ a_1+\ldots, +a_k: a_i\in A\}$ be the arithmetic sum of $A$ and $B$ and the $k$-iterated sum of $A$ respectively. A general principle in additive combinatorics is that if the arithmetic sum of a finite set $A$ with itself is ``small'', then the set $A$ itself has ``additive structure'', and in particular iterated sums and differences such as $A+A+A$ and $A-A$ are also ``small''. One precise formulation of this principle are the Pl\"{u}nnecke-Rusza inequalities, which say that if $A, B$ are two finite subsets of an abelian group and $|A+B|\le K |A|$, then
\[
|n B-m B|\le K^{n+m}|A|.
\]
In particular, taking $B=A$, this result gives a quantitative version of the above principle. The reader is referred to \cite{TaoVu2006} for the precise definitions and statements, as well as general background in additive combinatorics.

In this work we investigate whether similar statements can be made when $A$ is a subset of the real numbers (rather than the integers or a discrete group), and size is measured by some fractal dimension instead of cardinality. We were motivated in particular by the following question: if the Hausdorff dimension of $A+A$ is equal to the Hausdorff dimension of $A$, does it follow that the Hausdorff dimension of $A+A+A$ is also equal to the Hausdorff dimension of $A$? We prove that the answer is negative, in rather dramatic fashion: the sequence $\{\alpha_\ell\}_{\ell=1}^\infty$ of Hausdorff dimensions of the iterated sumsets $\ell A$ can be completely arbitrary, subject to the obvious restrictions of being nondecreasing and taking values in $[0,1]$. In particular, information about the Hausdorff dimension of the sumsets $A, 2A,\ldots, \ell A$ gives no information whatsoever about the Hausdorff dimension of $(\ell+1)A$, other than the trival fact that Hausdorff dimension is monotone. Thus, the fractal world exhibits very different behavior than the discrete world.

More generally, we investigate the possible simultaneous values of Hausdorff, lower box-counting and upper box-counting dimensions of iterated sumsets (The reader is referred to \cite{Falconer1990} for the definitions and basic properties of Hausdorff and box-counting dimensions). This turns out to be a delicate problem - controlling various dimensions at once is substantially harder than controlling just one of them.

We will denote Hausdorff dimension by $\dim_H$, and lower and upper box dimensions by $\ldim_B$ and $\udim_B$, respectively. The following is our main result:

\begin{theorem}\label{main}
Let $\{ \alpha_i\}_{i=1}^\infty$, $\{\beta_i\}_{i=1}^\infty$ and $\{\gamma_i\}_{i=1}^\infty$ be nondecreasing sequences with $0\le\alpha_i \le\beta_i\le \gamma_i\le 1$,
\[
\b_\ell\le\b_{\ell-1}+\b_1-\sum_{k=2}^{\ell-1}(\ell-k)(\b_{k-1}+\b_1 -\b_k)
\]
and
\[
\g_\ell\le\g_{\ell-1}+\g_1-\sum_{k=2}^{\ell-1}(\ell-k)(\g_{k-1}+\g_1 -\g_k).
\]
for all $\ell\ge 2$.

There exists a compact set $A\subset [0,1]$ such that
\[
\dim_H(\ell A) = \alpha_\ell,\qquad \ldim_B(\ell A)=\beta_\ell, \quad\textrm{ and }\quad \udim_B(\ell A)=\gamma_\ell
\]
for $\ell=1,\ldots, \infty$. Additionally, if $\alpha_{\ell}=1$ for some $\ell$, we can also require that $\ell A$ contains an interval.
\end{theorem}

Notice from the above result that, even if a set $A$ has coinciding Hausdorff, lower and upper box dimensions, it is possible that for the sumset $A+A$ all three concepts of dimension differ.

Since the construction of the set $A$ in Theorem \ref{main} is rather complicated, rather than giving a full proof we will present several examples of increasing complexity illustrating different features of the general construction. After these examples we indicate how to put them together to yield Theorem \ref{main}. This will be done in Section \ref{sec:examples}.

Theorem \ref{main} does not negate a result in the spirit of Pl\"{u}nnecke-Rusza for upper or lower box-counting dimension. In Section \ref{sec:plunnecke} we will show that there is indeed a natural extension of the Pl\"{u}nnecke-Rusza estimates for lower box dimension (but not for upper box dimension); see Proposition \ref{prop:positiveresult} for the precise quantitative estimates.


\section{Examples and proof of the main result} \label{sec:examples}

\subsection{Basic facts}

Before proving Theorem~\ref{main}, we will present some simpler but significant examples illustrating the main features of the construction. The construction itself is quite technical and will be sketched at the end of the section.

We will consider the numbers in the unit interval in their base 2 expansion, i.e. to a real number $x\in [0,1]$ we associate a binary infinite sequence $\ux=x_1x_2x_3\cdots $ such that $x_i\in\{0,1\}$ and $\displaystyle x=\sum_{i=1}^\infty\frac{x_i}{2^i}$. This sequence is unique unless $x$ is a dyadic rational, in which case we have exactly 2 representations. It will be apparent from the constructions that this will not affect the dimension calculations (in the case of Hausdorff dimension this is clear since countable sets have zero Hausdorff dimension).

We will use the following notation. If for a given set of sequences $x_1x_2\cdots $ the i-th symbol is not specified -- i.e. it can be chosen to be either 0 or 1 -- we will write $x_i=a$ ($a$ stands for ``arbitrary''). In all our constructions, the basic pieces of the set will be defined in terms of sequences which have $0$ at some positions and $a$ at the rest of the positions.

Before starting the constructions, we recall some basic properties of dimensions. Hausdorff dimension and upper box dimension are stable under finite unions, i.e.
\[
\dim\left(\bigcup_{i=1}^m A_i\right) = \max_{i=1}^m \dim(A_i),
\]
where $\dim$ stands for either $\dim_H$ or $\udim_B$. However, the lower box dimension $\ldim_B$ is \textit{not} stable under finite unions. These facts will be exploited repeatedly in our constructions.

Let $A\subset [0,1]$. For $\ux\in A$, we define
\[
\#_{off} (n,\ux, A):=\begin{cases} 1 & \text{ if } [x_1\cdots x_n a]\cap A\ne\emptyset \\
0 & \text{otherwise} \end{cases},
\]
and
\[
\on (A):=\min_{\ux\in A} \frac1n \sum_{i=1}^n \#_{off}(i,\ux, A).
\]
\begin{lemma}\label{lemma:on}
For any set $A\subset [0,1]$, we have
\[
\liminf_n\on (A)\le\dim_HA.
\]
\end{lemma}
\begin{proof}
Let $\log$ denote the logarithm to base $2$. Then for any measure $\mu$ on the space of binary sequences:
\[
d_\mu(\ux):=\liminf_{\e\to 0}\frac{\log\mu(B(\ux ,\e))}{\log\e}=\liminf_{n\to\infty}-\frac1n\log\mu (C_n(\ux)),
\]
where $C_n(\ux)$ denotes the set of infinite binary sequences starting with $x_1\cdots x_n$. Now let $\mu$ be the measure on $A$ that gives equal weight to any offspring of a given cylinder, i.e. it gives half of the measure if $\#_{off} (i,\ux, A)=1$ and full measure otherwise. Then for any $\ux\in A$ we have
\[
\liminf_{n\to\infty}-\frac1n\log\mu (C_n(\ux))\ge \liminf_n\on (A).
\]
An application of the mass distribution principle (see e.g. \cite[Chapter 4]{Falconer1990}) concludes the proof.
\end{proof}

\subsection{Examples for Hausdorff dimension}

The first example shows how one can control the Hausdorff dimension of simple sumsets.
\begin{example}\label{hdim1}
For $0\le\a_1\le\a_2\le 1$, we construct a compact set $A\subset [0,1]$ such that
\[
\dim_HA=\a_1\qquad \dim_H(A+A)=\a_2.
\]
\end{example}
\begin{proof}[Construction]
First we fix a sufficiently rapidly increasing sequence of natural numbers, say $n_k=2^{2^k}$. The set $A$ will be constructed as a (almost disjoint) union of two sets. Let
\[
A_1:=\left\{\ux \,:\, x_i=\begin{cases} 0 & i\in \left[n_{3k},\left[\frac{n_{3k}}{\a_1}\right]_*-1\right]\\ 0 & i\in \left[n_{3k+2},\left[\frac{n_{3k+2}}{\a_2}\right]_*-1\right]\\ a & \text{otherwise}  \end{cases}\right\},
\]
where $\left[\frac{n_i}{\a_j}\right]_*$ denotes the minimum of the integer part of $\frac{n_i}{\a_j}$ and $n_{i+1}-n_i$ if $\a_j\ne 0$, and $i n_i$ otherwise (we will use this notation for the rest of this construction and the next one). Note that, if $\a_j\neq 0$, then $\left[\frac{n_i}{\a_j}\right]_*$ equals the integer part of $\frac{n_i}{\a_j}$ for all but finitely many values of $i$. The second set is defined (only if $\a_2\ne 0$, otherwise it is empty) as
\[
A_2:=\left\{\ux \,:\, x_i=\begin{cases} 0 & i\in \left[n_{3k+1},\left[\frac{n_{3k+1}}{\a_1}\right]_*-1\right]\\ 0 & i\in \left[n_{3k+2},\left[\frac{n_{3k+2}}{\a_2}\right]_*-1\right]\\ a & \text{otherwise}  \end{cases}\right\}.
\]
If we consider cylinders of length $\left[\frac{n_{3k}}{\a_1}\right]_*$ we see that there are at most $\displaystyle 2^{n_{3k}}$ intersecting $A_1$, and if we consider cylinders of length $\left[\frac{n_{3k+1}}{\a_1}\right]_*$ there are at most $\displaystyle 2^{n_{3k+1}}$ intersecting $A_2$. Hence, $\a_1\ge\ldim_BA_i\ge\dim_HA_i$. On the other hand, $\liminf_n\on(A_i)=\a_1$ since $\a_2\ge\a_1$. Thus Lemma~\ref{lemma:on} gives the lower bound for $\dim_HA$.

For the simple sumset we can argue as follows. Firstly, we have that
\[
2A_1\cup 2A_2\cup (A_1+A_2)\subset\left\{\ux \,:\, x_i=\begin{cases} 0 & i\in \left[n_{3k+2},\left[\frac{n_{3k+2}}{\a_2}\right]_*-2\right]\\ a & \text{otherwise}  \end{cases}\right\}=:B_1
\]
where the ``$-2$'' accounts for a possible carry. Secondly,
\[
A_1+A_2\supset\left\{\ux \,:\, x_i=\begin{cases} 0 & i\in \left[n_{3k+2},\left[\frac{n_{3k+2}}{\a_2}\right]_*-1\right]\\ a & \text{otherwise}  \end{cases}\right\}=:B_2.
\]
For both sets on the right-hand-side we have that
\[
\a_2\le\liminf_n\on (B_i)\le\dim_HB_i\le\ldim_BB_i\le\a_2.
\]
This shows that the set $A=A_1\cup A_2$ has the desired properties.
\end{proof}

The second example shows how one can control the Hausdorff dimension of triple sumsets. This gives an idea about the general induction process.

\begin{example}\label{hdim2}
For $0\le\a_1\le\a_2\le\a_3\le 1$ we construct a compact set $A\subset [0,1]$ such that
\[
\dim_HA=\a_1\qquad \dim_H(A+A)=\a_2\qquad\dim_H(A+A+A)=\a_3.
\]
\end{example}
\begin{proof}[Construction] This example is a modification of the previous one. We just need to add a third component to control the triple sums.

Again we fix a sufficiently rapidly increasing sequence of natural numbers, say $n_k=2^{2^k}$. The set $A$ will be constructed as a (almost disjoint) union of three sets. Let
\[
A_1:=\left\{\ux \,:\, x_i=\begin{cases} 0 & i\in \left[n_{6k},\left[\frac{n_{6k}}{\a_1}\right]_*-1\right]\\ 0 & i\in \left[n_{6k+2},\left[\frac{n_{6k+2}}{\a_2}\right]_*-1\right]\\ 0 & i\in \left[n_{6k+4},\left[\frac{n_{6k+4}}{\a_2}\right]_*-1\right]\\  0 & i\in \left[n_{6k+5},\left[\frac{n_{6k+5}}{\a_3}\right]_*-1\right]\\a & \text{otherwise}  \end{cases}\right\},
\]
\[
A_2:=\left\{\ux \,:\, x_i=\begin{cases} 0 & i\in \left[n_{6k+1},\left[\frac{n_{6k+1}}{\a_1}\right]_*-1\right]\\ 0 & i\in \left[n_{6k+2},\left[\frac{n_{6k+2}}{\a_2}\right]_*-1\right]\\ 0 & i\in \left[n_{6k+3},\left[\frac{n_{6k+3}}{\a_2}\right]_*-1\right]\\  0 & i\in \left[n_{6k+5},\left[\frac{n_{6k+5}}{\a_3}\right]_*-1\right]\\a & \text{otherwise}  \end{cases}\right\},
\]
\[
A_3:=\left\{\ux \,:\, x_i=\begin{cases} 0 & i\in \left[n_{6k},\left[\frac{n_{6k}}{\a_1}\right]_*-1\right]\\ 0 & i\in \left[n_{6k+3},\left[\frac{n_{6k+3}}{\a_2}\right]_*-1\right]\\ 0 & i\in \left[n_{6k+4},\left[\frac{n_{6k+4}}{\a_2}\right]_*-1\right]\\  0 & i\in \left[n_{6k+5},\left[\frac{n_{6k+5}}{\a_3}\right]_*-1\right]\\a & \text{otherwise}  \end{cases}\right\}.
\]
By a similar reasoning as in the previous example, we have that $\dim_HA_1=\dim_HA_2=\dim_HA_3=\dim_HA=\a_1$, where $A=A_1 \cup A_2 \cup A_3$. Also for the simple sumsets we have that
\begin{align*}
\dim_H (A_i+A_i)&\le\dim_H(A_1+A_2)=\dim_H(A_2+A_3)=\dim_H(A_1+A_3)\\&=\dim_H(A+A)=\a_2.
\end{align*}
For the triple sumset, we remark that
\[
\bigcup_{i,j,l=1,2,3}(A_i+A_j+A_l)\subset\left\{\ux \,:\, x_i=\begin{cases} 0 & i\in \left[n_{6k+5},\left[\frac{n_{6k+5}}{\a_3}\right]_*-3\right]\\ a & \text{otherwise}  \end{cases}\right\}=:B_1
\]
where the ``$-3$'' accounts for carryovers. Secondly,
\[
A_1+A_2+A_3\supset\left\{\ux \,:\, x_i=\begin{cases} 0 & i\in \left[n_{6k+5},\left[\frac{n_{6k+5}}{\a_3}\right]_*-1\right]\\ a & \text{otherwise}  \end{cases}\right\}=:B_2.
\]
Again for both sets on the right-hand-side we have that
\[
\a_3\le\liminf_n\on (B_i)\le\dim_HB_i\le\ldim_BB_i\le\a_3.
\]
This shows that the set $A=A_1\cup A_2\cup A_3$ has the desired properties.
\end{proof}

The previous example can clearly be generalized to control the Hausdorff dimension of any finite sumsets, i.e. of sets $A$, $2A$, $\cdots$, $\ell A$ for any $\ell\in\NN$. Moreover, it is easy to see that in Example \ref{hdim2}, if some $\alpha_i=1$ for $i=1,2,3$, then $i A\supset [0,1]$, since the choice of digits becomes completely arbitrary.  Next, we show how to control an infinite number of sumsets.

\begin{example} \label{hausdiminfinite}
Let $\{\alpha_i\}_{i=1}^\infty$ be a non-decreasing sequence taking values in $[0,1]$. Then there exists a compact set $A\subset [0,1]$ such that
\[
\dim_H(\ell A) = \alpha_\ell
\]
for all $\ell\in\mathbb{N}$. Furthermore, if $\alpha_\ell=1$ for some $\ell$, we can also require that $\ell A$ contains an interval.
\end{example}
\begin{proof}[Construction]
In short, the construction consists in pasting together along the dyadic structure all the sets obtained for finite sequences $\a_1,\ldots, \a_\ell$. By proceeding like in Example \ref{hdim2}, we see that for every $\ell\ge 2$, there exists a compact set $A_\ell\subset [0,1]$ such that
\[
\dim_H(i A_\ell) = \alpha_i, \quad i=1,\ldots,\ell.
\]
Let $\{M_\ell\}_{\ell=1}^\infty$ be a rapidly increasing sequence. Set $S_1 = 0$ and $S_\ell = \sum_{i=1}^{\ell-1} M_\ell$ for $\ell>1$, and let
\[
A := \left\{ \ux \,:\, x_{S_\ell+1} \cdots x_{S_{\ell+1}} = y_1 \ldots y_{M_\ell} \text{ for some } \underline{y} \in A_\ell, \, \text{ for each } \ell \in\NN  \right\}.
\]
Roughly speaking, $A$ is defined by following the construction of $A_1$ for the first $M_1$ binary digits, then the construction of $A_2$ for the following $M_2$ binary digits, and so on. Note that $A$ is a countable intersection of compact sets, so it is compact.

We claim that
\[
\dim_H(A) = \liminf_i \dim_H(A_i) = \a_1,
\]
provided $\{M_\ell\}$ grows fast enough. Indeed, by taking $M_\ell$ large enough, we can cover each $A_\ell$ by dyadic intervals $\{ I_\ell^{(r)}\}$ of length at least $2^{-M_\ell}$, satisfying
\[
\sum_r \left|I_\ell^{(r)}\right|^{\a_\ell+1/\ell} < 1.
\]
Then we can cover $A$ by at most $2^{S_i\ell}$ translated and scaled down (by a factor of $2^{-S_\ell}$) copies of the family $\{ I_\ell^{(r)}\}$. Since $M_\ell \gg S_\ell$, this yields the upper bound for $\dim_H(A)$. For the lower bound, we use Frostman's Lemma: there are measures $\mu_\ell$ supported on $A_\ell$, such that
\begin{equation} \label{eq:frostman}
\mu_\ell(I) \le |I|^{\a_\ell-1/\ell},
\end{equation}
for all dyadic intervals of length $|I|\le 2^{-M'_\ell}$, where $M'_\ell \ll M_{\ell-1}$ (making $M_{\ell-1}$ larger if necessary). We can paste all these measures together dyadically in a similar way to the construction of $A$. More precisely,
\begin{equation} \label{eq:pastedmeasure}
\mu(C(x_1\cdots x_{S_\ell})) := \mu_1(C(x_1\cdots x_{M_1})) \cdots \mu_i(C(x_{S_{\ell-1}+1}\cdots x_{S_\ell})),
\end{equation}
where $C(y_1\cdots y_j)$ denotes the set of all dyadic sequences starting with $y_1\cdots y_j$. Combining \eqref{eq:frostman} and \eqref{eq:pastedmeasure} and applying the mass distribution principle yields the lower bound, completing the proof of the claim.

More generally, since addition preserves the dyadic structure except for carryovers, and these are negligible due to the presence of blocks of zeroes in the construction of each $A_\ell$ (unless $\a_\ell=1$), we see that
\[
\dim_H(\ell A) = \liminf_i \dim_H(\ell A_i) = \a_\ell.
\]
Finally, if $\a_\ell=1$ for some $\ell$, then there is no restriction on the dyadic digits of $\ell A_i$ for any $i\ge \ell$, thus there is no restriction on the dyadic digits of $\ell A$ except for finitely many of them. Hence $\ell A$ contains an interval, as desired.
\end{proof}

\subsection{Examples for Hausdorff and box-counting dimensions}

Next, we start controlling various notions of dimension simultaneously. In the first example of this kind we show how to control the Hausdorff and lower box-counting dimension for simple sumsets.

\begin{example}\label{hauslowbox}
Given $0\le \a_i \le b_i \le 1$, $i=1,2$ with $\a_1\le a_2$ and $\b_1\le \b_2 \le 2\b_1$, we construct a compact set $A\subset[0,1]$ with
\[
\dim_H A = \a_1 \qquad \ldim_B A = \b_1,
\]
and
\[
\dim_H(A+A) = \a_2 \qquad \ldim_B(A+A) = \b_2,
\]
\end{example}
\begin{proof}[Construction]
Again we fix a fast increasing sequence
\begin{equation} \label{eq:fastsequence}
n_k = \min\left\{n\in\NN\, :\, n\ge 2^{2^k}\text{ and } (k-1)\vert (n_k-n_{k-1})\right\}.
\end{equation}
For each $k\in\NN$ we define four numbers:
\[
l_k:= [k \b_1 ], \qquad m_k := [k \b_2],
\]
\[
d_i(k):=\begin{cases}
k\left[\frac1kn_k\left(\frac{\beta_i}{\a_i}-1\right)\right] & \text{ if } \a_i\neq 0\\
k n_k & \text{ if } \a_i=0
\end{cases}\,, \quad i=1,2.
\]
Given a word $u$, we let $u^r$ denote the word consisting of $r$ consecutive copies of $u$; in particular, $a^r$ is the word consisting of $r$ consecutive ``arbitrary symbols'' $a$. If $r=0$, then $u^r$ is the empty word. We will define four types of blocks:
\begin{align*}
t_{\a_1}(k):=\Big\{x_{n_k}&\cdots x_{n_{k+1}-1}\, :\, x_{n_k}=\cdots x_{n_k+d_1(k)-1}=0,\\& x_{n_k+d_1(k)}\cdots x_{n_{k+1}-1}=\left(a^{l_k}0^{k-l_k}\right)^{(n_{k+1}-n_k-d_1(k))/k}\Big\},
\end{align*}
\begin{align*}
t_{\a_2}(k):=\Big\{x_{n_k}&\cdots x_{n_{k+1}-1}\, :\, x_{n_k}=\cdots x_{n_k+d_2(k)-1}=0,\\& x_{n_k+d_2(k)}\cdots x_{n_{k+1}-1}=\left(0^{m_k-l_k}a^{l_k}0^{k-m_k}\right)^{(n_{k+1}-n_k-d_2(k))/k}\Big\},
\end{align*}
\[
t_{\b_1}(k):=\left\{x_{n_k}\cdots x_{n_{k+1}-1}\, :\, x_{n_k}\cdots x_{n_{k+1}-1}=\left(a^{l_k}0^{k-l_k}\right)^{(n_{k+1}-n_k)/k}\right\},
\]
\[
t_{\b_2}(k):=\left\{x_{n_k}\cdots x_{n_{k+1}-1}\, :\, x_{n_k}\cdots x_{n_{k+1}-1}=\left(0^{m_k-l_k}a^{l_k}0^{k-m_k}\right)^{(n_{k+1}-n_k)/k}\right\}.
\]
The set $A$ will be defined as the union of six components $A_i$, which are defined as follows:
\[
A_1:=\left\{\ux \, :\, \begin{cases} x_{n_{3k}}\cdots x_{n_{3k+1}-1}\in t_{\a_1}(3k)\\
         x_{n_{3k+1}}\cdots x_{n_{3k+2}-1}\in t_{\a_2}(3k+1)\\
         x_{n_{3k+2}}\cdots x_{n_{3(k+1)}-1}\in t_{\b_1}(3k+2)
\end{cases}\right\},
\]
\[
A_2:=\left\{\ux \, :\, \begin{cases} x_{n_{3k}}\cdots x_{n_{3k+1}-1}\in t_{\b_1}(3k)\\
         x_{n_{3k+1}}\cdots x_{n_{3k+2}-1}\in t_{\a_1}(3k+1)\\
         x_{n_{3k+2}}\cdots x_{n_{3(k+1)}-1}\in t_{\a_2}(3k+2)
\end{cases}\right\},
\]
\[
A_3:=\left\{\ux \, :\, \begin{cases} x_{n_{3k}}\cdots x_{n_{3k+1}-1}\in t_{\a_2}(3k)\\
         x_{n_{3k+1}}\cdots x_{n_{3k+2}-1}\in t_{\b_1}(3k+1)\\
         x_{n_{3k+2}}\cdots x_{n_{3(k+1)}-1}\in t_{\a_1}(3k+2)
\end{cases}\right\},
\]
\[
A_4:=\left\{\ux \, :\, \begin{cases} x_{n_{3k}}\cdots x_{n_{3k+1}-1}\in t_{\a_2}(3k)\\
         x_{n_{3k+1}}\cdots x_{n_{3k+2}-1}\in t_{\a_1}(3k+1)\\
         x_{n_{3k+2}}\cdots x_{n_{3(k+1)}-1}\in t_{\b_2}(3k+2)
\end{cases}\right\},
\]
\[
A_5:=\left\{\ux \, :\, \begin{cases} x_{n_{3k}}\cdots x_{n_{3k+1}-1}\in t_{\b_2}(3k)\\
         x_{n_{3k+1}}\cdots x_{n_{3k+2}-1}\in t_{\a_2}(3k+1)\\
         x_{n_{3k+2}}\cdots x_{n_{3(k+1)}-1}\in t_{\a_1}(3k+2)
\end{cases}\right\},
\]
\[
A_6:=\left\{\ux \, :\, \begin{cases} x_{n_{3k}}\cdots x_{n_{3k+1}-1}\in t_{\a_1}(3k)\\
         x_{n_{3k+1}}\cdots x_{n_{3k+2}-1}\in t_{\b_2}(3k+1)\\
         x_{n_{3k+2}}\cdots x_{n_{3(k+1)}-1}\in t_{\a_2}(3k+2)
\end{cases}\right\}.
\]
Note that block structure of $A_1, A_2$ and $A_3$ follows a cyclic pattern, and that the block structure of $A_{3+i}$ is specular to that of $A_i$, in the sense that blocks $t_{*_i}(k)$ are replaced by $t_{*_{2-i}}(k)$, for $*=\a,\b$.

We will use the following notation. Given a finite word $u$, by $\fr(u)$ we denote the frequency of symbols 'a' in $u$, i.e.
\[
\fr(u)=\frac{|\{ i: u_i=a\}|}{|u|},
\]
where $|u|$ is the length of $u$. A quantity that goes to zero as $k\rightarrow\infty$ will be denoted by $o(1)$.

Notice that all types of blocks have a frequency $\beta_1+o(1)$ of 'a' symbols. Since the blocks $t_{\a_1}(k)$ start with $d_1(k)$ zeroes, we see that if $k=3 l$, then
\[
\of_{n_k+d_1(k)}(A_1) = \a_1+o(1),
\]
and $\on(A_1) \ge \a_1$ for all $n$, so by Lemma \ref{lemma:on} we get
\[
\a_1 \le \dim_H A_1 \le \ldim_B A_1 \le \a_1.
\]
The same argument holds for the other $A_i$, so we obtain $\dim_H A=\a_1$. Next, notice that for each $k$ there is at least one component (in fact two) $A_i$ for which the block $x_{n_k}\cdots x_{n_{k+1}-1}$ is neither of type $t_{\a_1}(k)$ nor $t_{\a_2}(k)$. This shows that $\ldim_B A= \b_1$.

Let us write
\[
t_*(k)+t_{**}(k):=\{x_{n_k}\cdots x_{n_{k+1}-1} = v + w \, :\, v\in t_*(k),\,\, w\in t_{**}(k)\},
\]
where $a+a=a+0=a$ (thus carryovers are ignored, but due to the structure of the blocks they are negligible in the estimates). Note that
\begin{equation} \label{eq:frequencyb2}
\fr(t_*(k)+t_{**}(k)) \le \b_2+o(1)
\end{equation}
for any choice of $*, **$, and therefore $\udim_B(A_i) \le \b_2$. On the other hand, each $A_i+A_j$ contains blocks of either type $t_{\a_1}(k)+t_{\a_2}(k)$ or $t_{\a_2}(k)+t_{\a_2}(k)$. Taking into account the definition of $d_2(k)$, we see that $\dim_H(A_i+A_j) \le \a_2$ for each $i,j$. In the opposite direction, note that since $\b_2\le 2\b_1$, we have
\[
\fr(t_{\b_1}(k)+t_{\b_2}(k))=\b_2+o(1).
\]
Since $A_1+A_4$ has infinitely many blocks $t_{\a_1}(k)+t_{\a_2}(k)$ preceded by $t_{\b_1}(k-1)+t_{\b_2}(k-1)$, and it has no blocks of the form $t_{\a_1}(k)+t_{\a_1}(k)$ nor blocks of the form $t_{\b_1}+t_{\b_1}$, we see that $\liminf_n \on(A_1+A_4) \ge \a_2$. Hence we have shown that
\[
\dim_H(A+A) = \max_{i,j} \dim_H(A_i+A_j) = \dim_H(A_1+A_4) = \a_2.
\]
By \eqref{eq:frequencyb2}, $\udim_B (A_i+A_j)  \le \b_2$, so that
\[
\ldim_B (A+A)\le \udim_B (A+A)= \max_{i,j} \udim_B (A_i+A_j) \le  \b_2.
\]
On the other hand, for each $k$, there is $i\in\{1,2,3\}$ such that the block $x_{n_k}\cdots x_{n_{k+1}-1}$ in $A_i+A_{3+i}$ is of type $t_{\b_1}(k)+t_{\b_2}(k)$; all of these have frequency $\b_2+o(1)$ of 'a's distributed in small (relative to $n_k$) chunks of length $k$, so $\ldim_B(A+A) \ge \b_2$, as desired.
\end{proof}

The next example shows how to control Hausdorff, lower box-counting and upper box-counting dimension at once in simple sumsets.

\begin{example}\label{alldim2}
Given $0\le\a_i\le\b_i\le\g_i\le 1$, $i=1,2$ with $\a_1\le \a_2, \b_1\le \b_2, \g_1\le \g_2$, $\b_2\le 2\b_1$ and $\g_2\le 2\g_1$, we construct a compact set $A\subset [0,1]$ with
\[
\dim_HA=\a_1\qquad\ldim_BA=\b_1\qquad\udim_BA=\g_1,
\]
and
\[
\dim_H(A+A)=\a_2\qquad\ldim_B(A+A)=\b_2\qquad\udim_B(A+A)=\g_2.
\]
In particular if $\a_1=\b_1=\g_1$ the regularity of the set $A$ does not imply the regularity of the sumset $A+A$.
\end{example}
\begin{proof}[Construction] This example will be a modification of the previous one (one can check that in Example \ref{hauslowbox}, $\ldim_B(A)=\udim_B(A)$ and $\ldim_B(A+A)=\udim_B(A+A)$). All we need to do is to add in the construction of each $A_i$ blocks of type $t_{\gamma_1}(k)$ or $t_{\gamma_2}(k)$, at the same positions in each $A_i$, and preceding any blocks of the form $t_{\a_1}(k)$ or $t_{\a_2}(k)$ (to prevent the Hausdorff dimension from dropping too much).

We use the same fast increasing sequence defined in \eqref{eq:fastsequence}. For each $k\in\NN$ we define six numbers. The numbers $l_k, m_k$ are defined exactly like in Example \ref{hauslowbox}, while the numbers $d_i(k)$ are redefined as
\[
d_i(k):= \begin{cases}
k\left[\frac1kn_k\left(\frac{\gamma_i}{\a_i}-1\right)\right] & \text{ if } \a_i\neq 0\\
k n_k & \text{ if } \a_i =0
\end{cases}.
\]
Additionally, we define
\[
p_k:= [ k \g_1 ], \qquad q_k := [ k\g_2].
\]
We will use 6 types of blocks; the blocks $t_{\a_i}(k), t_{\b_i}(k)$, $i=1,2$ are defined just like in the previous example. We additionally define:
\[
t_{\gamma_1}(k):=\left\{x_{n_k}\cdots x_{n_{k+1}-1}\, :\, x_{n_k}\cdots x_{n_{k+1}-1}=\left(a^{p_k}0^{k-p_k}\right)^{(n_{k+1}-n_k)/k}\right\},
\]
\[
t_{\gamma_2}(k):=\left\{x_{n_k}\cdots x_{n_{k+1}-1}\, :\, x_{n_k}\cdots x_{n_{k+1}-1}=\left(0^{q_k-p_k}a^{p_k}0^{k-q_k}\right)^{(n_{k+1}-n_k)/k}\right\}.
\]

The set $A$ will have 6 components $A_i$. The first three are defined as:
\[
A_1:=\left\{\ux \, :\, \begin{cases}
         x_{n_{6k}}\cdots x_{n_{6k+1}-1}\in t_{\g_1}(6k)\\
         x_{n_{6k+1}}\cdots x_{n_{6k+2}-1}\in t_{\a_1}(6k+1)\\
         x_{n_{6k+2}}\cdots x_{n_{6k+3}-1}\in t_{\g_2}(6k+2)\\
         x_{n_{6k+3}}\cdots x_{n_{6k+4}-1}\in t_{\a_2}(6k+3)\\
         x_{n_{6k+4}}\cdots x_{n_{6k+5}-1}\in t_{\g_1}(6k+4)\\
         x_{n_{6k+5}}\cdots x_{n_{6(k+1)}-1}\in t_{\b_1}(6k+5)\\
\end{cases}\right\},
\]
\[
A_2:=\left\{\ux \, :\, \begin{cases}
         x_{n_{6k}}\cdots x_{n_{6k+1}-1}\in t_{\g_1}(6k)\\
         x_{n_{6k+1}}\cdots x_{n_{6k+2}-1}\in t_{\b_1}(6k+1)\\
         x_{n_{6k+2}}\cdots x_{n_{6k+3}-1}\in t_{\g_1}(6k+2)\\
         x_{n_{6k+3}}\cdots x_{n_{6k+4}-1}\in t_{\a_1}(6k+3)\\
         x_{n_{6k+4}}\cdots x_{n_{6k+5}-1}\in t_{\g_2}(6k+4)\\
         x_{n_{6k+5}}\cdots x_{n_{6(k+1)}-1}\in t_{\a_2}(6k+5)\\
\end{cases}\right\},
\]
\[
A_3:=\left\{\ux \, :\, \begin{cases}
         x_{n_{6k}}\cdots x_{n_{6k+1}-1}\in t_{\g_2}(6k)\\
         x_{n_{6k+1}}\cdots x_{n_{6k+2}-1}\in t_{\a_2}(6k+1)\\
         x_{n_{6k+2}}\cdots x_{n_{6k+3}-1}\in t_{\g_1}(6k+2)\\
         x_{n_{6k+3}}\cdots x_{n_{6k+4}-1}\in t_{\b_1}(6k+3)\\
         x_{n_{6k+4}}\cdots x_{n_{6k+5}-1}\in t_{\g_1}(6k+4)\\
         x_{n_{6k+5}}\cdots x_{n_{6(k+1)}-1}\in t_{\a_1}(6k+5)\\
\end{cases}\right\}.
\]
Note that the only difference with the sets in Example \ref{hauslowbox} is the addition of blocks corresponding to the upper box dimension. Likewise, for $i\in\{1,2,3\}$, the sets $A_{i+3}$ are defined in a specular way to $A_i$ like in the previous example, namely blocks of type $t_{*_i}(k)$ are replaced by blocks of type $t_{*_{2-i}}(k)$ for $*=\a,\b$ and $\g$.

One can check that $\dim_H A=\a_1$, $\ldim_B A=\b_1$ just like in Example \ref{hauslowbox} (for the Hausdorff dimension, it is useful to note that  blocks of type $t_{\a_1}(k)$ are always preceded by blocks of type $t_{\g_1}(k-1)$). Since
\[
\fr(t_{\g_i}(k)) = \g_1+o(1),
\]
for $i=1,2$, it follows that $\udim_B(A)=\g_1$.

For the sumset we argue just like in Example \ref{hauslowbox}. For the upper box dimension, all we need to observe is that
\[
\fr(t_{\gamma_1}(k)+t_{\gamma_2}(k)) = \gamma_2 + o(1),
\]
and such blocks occur infinitely often in $A_1+A_4$; any other block $t_*(k)+t_{**}(k)$ has a lower frequency of 'a's. Thus
\[
\dim_H(A+A)=\a_2\qquad \ldim_B(A+A)=\b_2\qquad \udim_B(A+A)=\g_2.
\]
\end{proof}

\begin{example}\label{alldim3}
Suppose $0\le\a_i\le\b_i\le\g_i\le 1$, $i=1,2,3$, and
\[
\b_2\le 2\b_1 \qquad \b_3\le 2\b_2-\b_1,
\]
\[
\g_2\le 2\g_1 \qquad \g_3\le 2\g_2-\g_1.
\]
We construct a compact set $A\subset [0,1]$ with
\[
\dim_HA=\a_1\qquad\ldim_BA=\b_1\qquad\udim_BA=\g_1,
\]
\[
\dim_H(A+A)=\a_2\qquad\ldim_B(A+A)=\b_2\qquad\udim_B(A+A)=\g_2,
\]
and
\[
\dim_H(A+A+A)=\a_3\qquad\ldim_B(A+A+A)=\b_3\qquad\udim_B(A+A+A)=\g_3.
\]
\end{example}
\begin{proof}[Construction]
We fix again the fast increasing sequence $n_k$ given by \eqref{eq:fastsequence}. For each $k\in\NN$ we will define nine numbers. The numbers $l_k, m_k, p_k$ and $q_k$ are defined exactly like in Example \ref{alldim2}. The numbers $d_i(k)$ are also defined in the same way, except that now we also allow the index $i=3$. We also define new numbers:
\[
s_k:= [ k\b_3 ],\qquad v_k := [k\g_3].
\]
We will define nine types of blocks. The blocks $t_{\a_i}(k), t_{\b_i}(k)$ and $t_{\g_i}(k)$, $i=1,2$, are the same as in Example \ref{alldim2}. We further define:
\begin{align*}
t_{\a_3}(k):=\Big\{x_{n_k}&\cdots x_{n_{k+1}-1}\, :\, x_{n_k}=\cdots x_{n_k+d_3(k)-1}=0, \quad\\& x_{n_k+d_3(k)}\cdots x_{n_{k+1}-1}=\\&\left(0^{m_k-l_k}a^{l_k+m_k-s_k}0^{s_k-m_k}a^{s_k-m_k}0^{k-s_k}\right)^{(n_{k+1}-n_k-d_3(k))/k}\Big\},
\end{align*}
\begin{align*}
t_{\b_3}(k):=\Big\{x_{n_k}\cdots &x_{n_{k+1}-1}\, :\, x_{n_k}\cdots x_{n_{k+1}-1}=\\&\left(0^{m_k-l_k}a^{l_k+m_k-s_k}0^{s_k-m_k}a^{s_k-m_k}0^{k-s_k}\right)^{(n_{k+1}-n_k)/k}\Big\},
\end{align*}
\begin{align*}
t_{\g_3}(k):=\Big\{x_{n_k}\cdots &x_{n_{k+1}-1}\, :\, x_{n_k}\cdots x_{n_{k+1}-1}=\\&\left(0^{q_k-p_k}a^{p_k+q_k-v_k}0^{v_k-q_k}a^{v_k-q_k}0^{k-v_k}\right)^{(n_{k+1}-n_k)/k}\Big\}.
\end{align*}
It is easy to verify that $t_{\b_3}(k)$ and $t_{\g_3}(k)$ are well-defined due to the assumptions made on the $\b_i$ and $\g_i$.

The set $A$ will have 18 components $A_i$, which are defined by specifying the types of blocks
\[
x_{n_k}\cdots x_{n_{k+1}-1} = t_*(k)
\]
as follows:
\[
\begin{tabular}{|c|cccccc|cccccc|}
\hline
12k+ &  0 &  1 &   2 &  3 & 4 & 5 & 6 & 7 & 8 & 9 & 10 & 11 \\
\hline
$A_1$ & $\g_1$ & $\a_1$ & $\g_2$ & $\a_2$ & $\g_3$ & $\a_3$ & $\g_1$ & $\a_1$ & $\g_2$ & $\a_2$ & $\g_3$ & $\b_1$ \\
$A_2$ & $\g_1$ & $\a_1$ & $\g_2$ & $\a_2$ & $\g_3$ & $\a_3$ & $\g_1$ & $\a_1$ & $\g_2$ & $\b_1$ & $\g_3$ & $\a_3$ \\
$A_3$ & $\g_1$ & $\a_1$ & $\g_2$ & $\a_2$ & $\g_3$ & $\a_3$ & $\g_1$ & $\b_1$ & $\g_2$ & $\a_2$ & $\g_3$ & $\a_3$ \\
$A_4$ & $\g_1$ & $\a_1$ & $\g_2$ & $\a_2$ & $\g_3$ & $\b_1$ & $\g_1$ & $\a_1$ & $\g_2$ & $\a_2$ & $\g_3$ & $\a_3$ \\
$A_5$ & $\g_1$ & $\a_1$ & $\g_2$ & $\b_1$ & $\g_3$ & $\a_3$ & $\g_1$ & $\a_1$ & $\g_2$ & $\a_2$ & $\g_3$ & $\a_3$ \\
$A_6$ & $\g_1$ & $\b_1$ & $\g_2$ & $\a_2$ & $\g_3$ & $\a_3$ & $\g_1$ & $\a_1$ & $\g_2$ & $\a_2$ & $\g_3$ & $\a_3$ \\
\hline
$A_7$ & $\g_2$ & $\a_2$ & $\g_3$ & $\a_3$ & $\g_1$ & $\a_1$ & $\g_3$ & $\a_3$ & $\g_1$ & $\a_1$ & $\g_2$ & $\b_2$ \\
$A_8$ & $\g_2$ & $\a_2$ & $\g_3$ & $\a_3$ & $\g_1$ & $\a_1$ & $\g_3$ & $\a_3$ & $\g_1$ & $\b_2$ & $\g_2$ & $\a_2$ \\
$A_9$ & $\g_2$ & $\a_2$ & $\g_3$ & $\a_3$ & $\g_1$ & $\a_1$ & $\g_3$ & $\b_2$ & $\g_1$ & $\a_1$ & $\g_2$ & $\a_2$ \\
$A_{10}$ & $\g_2$ & $\a_2$ & $\g_3$ & $\a_3$ & $\g_1$ & $\b_2$ & $\g_3$ & $\a_3$ & $\g_1$ & $\a_1$ & $\g_2$ & $\a_2$ \\
$A_{11}$ & $\g_2$ & $\a_2$ & $\g_3$ & $\b_2$ & $\g_1$ & $\a_1$ & $\g_3$ & $\a_3$ & $\g_1$ & $\a_1$ & $\g_2$ & $\a_2$ \\
$A_{12}$ & $\g_2$ & $\b_2$ & $\g_3$ & $\a_3$ & $\g_1$ & $\a_1$ & $\g_3$ & $\a_3$ & $\g_1$ & $\a_1$ & $\g_2$ & $\a_2$ \\
\hline
$A_{13}$ & $\g_3$ & $\a_3$ & $\g_1$ & $\a_1$ & $\g_2$ & $\a_2$ & $\g_2$ & $\a_2$ & $\g_3$ & $\a_3$ & $\g_1$ & $\b_3$ \\
$A_{14}$ & $\g_3$ & $\a_3$ & $\g_1$ & $\a_1$ & $\g_2$ & $\a_2$ & $\g_2$ & $\a_2$ & $\g_3$ & $\b_3$ & $\g_1$ & $\a_1$ \\
$A_{15}$ & $\g_3$ & $\a_3$ & $\g_1$ & $\a_1$ & $\g_2$ & $\a_2$ & $\g_2$ & $\b_3$ & $\g_3$ & $\a_3$ & $\g_1$ & $\a_1$ \\
$A_{16}$ & $\g_3$ & $\a_3$ & $\g_1$ & $\a_1$ & $\g_2$ & $\b_3$ & $\g_2$ & $\a_2$ & $\g_3$ & $\a_3$ & $\g_1$ & $\a_1$ \\
$A_{17}$ & $\g_3$ & $\a_3$ & $\g_1$ & $\b_3$ & $\g_2$ & $\a_2$ & $\g_2$ & $\a_2$ & $\g_3$ & $\a_3$ & $\g_1$ & $\a_1$ \\
$A_{18}$ & $\g_3$ & $\b_3$ & $\g_1$ & $\a_1$ & $\g_2$ & $\a_2$ & $\g_2$ & $\a_2$ & $\g_3$ & $\a_3$ & $\g_1$ & $\a_1$ \\
\hline
\end{tabular}
\]
Arguing like in the previous examples we see that
\[
\dim_H A = \a_1 \qquad \ldim_B A = \b_1 \qquad \udim_B A = \g_1.
\]
For the sumset $A+A$, notice that each pair $A_i+A_j$ contains infinitely many blocks of the form $t_{\a_*}(k)+ t_{\a_{**}}(k)$, where $*$ and $**$ are either $1$ or $2$ (possibly $*=**$). Hence, again arguing like in the previous constructions, $\dim_H(A) \le \a_2$. On the other hand, for $A_1+A_7$ all blocks of the form $t_{\a_1}(k)+t_{\a_2}(k)$ are preceded by blocks of the form $t_{\g_1}(k-1)+t_{\g_2}(k-1)$, and there are infinitely many such blocks.  Also, there are no blocks of the form $t_{*_1}+ t_{*_1}$ for $*=\a,\b$ or $\g$. From this we deduce that
\[
\dim_H(A+A)\ge \dim_H(A_1+A_7) \ge \a_2.
\]
For the lower box dimension, we note that for each $k$, there is an $i\in \{1,\ldots, 7\}$ such that $A_i + A_{i+6}$ contains a block of type either $t_{\b_1}(k)+ t_{\b_2}(k)$ or $t_{\g_1}(k)+ t_{\g_2}(k)$, so that $\ldim_B(A+A) \ge \b_2$. On the other hand, if $k= 12 l+1$, then
\[
\fr(t_*(k)+t_{**}(k)) \le \b_2,
\]
for all possible occurrences of $*$ and $**$ (to see this in the case $*=\b_1$ or $\b_2$ and $**=\b_3$, one needs to make use of the assumption $\b_3 \le 2\b_2-\b_1$). Hence $\ldim_B(A+A) \le \b_2$. It is easy to check that $\udim_B(A+A) = \g_2$.

Finally, if we consider $A+A+A$, we see that for each $i, j, k$ there are infinitely many blocks of the form $t_{\a_*}(k)+ t_{\a_{**}}(k)+t_{\a_{***}}(k)$ in $A_i+ A_j + A_k$, where $*, **, ***\in \{1,2,3\}$. This implies that $\dim_H(A+A+A) \le \a_3$. On the other hand, $A_1+A_7+A_{13}$ contains infinitely many blocks of the form $t_{\a_1}(k)+t_{\a_2}(k)+t_{\a_3}(k)$, all of them preceded by blocks of the form $t_{\g_1}(k-1)+t_{\g_2}(k-1)+t_{\g_3}(k-1)$; moreover, any block in $A_1+A_7+A_{13}$ is of the form $t_{*_1}(k)+ t_{**_2}(k) + t_{***_3}(k)$. It follows that $\dim_H(A+A+A) \ge \a_3$.

The arguments for $\ldim_B(A+A+A)$ and $\udim_B(A+A+A)$ are just like for the sums $A+A$.

\end{proof}

\subsection{Proof of the main result}

\begin{proof}[Proof of Theorem \ref{main}]
The result is proved by combining the examples above, in particular Examples \ref{hausdiminfinite} and \ref{alldim3}. Since the actual construction is quite complicated, we sketch the main ideas, leaving the details to the interested reader.

Example \ref{alldim3} can be generalized to $\ell$-sumsets in a straightforward way if $\ell$ is a prime number (and there is no loss of generality in assuming this). We need $(\ell-1)\ell^2$ components $A_i$, $3\ell$ different types of blocks, and we have to control the sequences for $2(\ell-1)\ell$ consecutive $n_k$'s. The restrictions on the dimensions that arise are precisely those stated in the theorem.

Given $\ell\ge 2$, let $A_\ell\subset [0,1]$ be a compact set such that
\[
\dim_H(i A_\ell) = \a_i,\qquad \ldim_B(i A_\ell) = \b_i,\qquad \udim_B(i A_\ell) = \g_i,
\]
for $i=1,\ldots,\ell$.

We are going to combine the sets $A_\ell$ exactly like in Example \ref{hausdiminfinite}. Let $\{M_\ell\}_{\ell=1}^\infty$ be a rapidly increasing sequence. Write $S_\ell = \sum_{i=1}^\ell M_i$, with $S_1=0$, and define
\[
A := \left\{ \ux \,:\, x_{S_\ell+1} \cdots x_{S_{\ell+1}} = y_1 \ldots y_{M_\ell} \text{ for some } \underline{y} \in A_\ell, \, \text{ for each } \ell \in\NN  \right\}.
\]
The set $A$ is clearly compact, and its structure translates to any finite sumset $\ell A$, apart from carryovers which can be ignored. Therefore we get
\begin{align*}
\dim_H(\ell A) &= \liminf_{i\rightarrow\infty} \dim_H(\ell A_i) = \a_\ell,\\
\ldim_B(\ell A) &= \liminf_{i\rightarrow\infty} \ldim_B(\ell A_i) = \b_\ell,\\
\udim_B(\ell A) &= \limsup_{i\rightarrow\infty} \udim_B(\ell A_i) = \g_\ell,
\end{align*}
provided $M_\ell$ grows quickly enough (this was proved for Hausdorff dimension in Example \ref{hausdiminfinite}; the proof for box dimensions is similar but easier). If $\a_\ell=1$ for some $\ell$, then obviously $\b_\ell=\g_\ell=1$ as well, and one can check that there is no restriction in the digits of $\ell A_i$ for all $i\ge \ell$. Thus $A$ has no restriction on all but finitely many of its binary digits, and therefore it contains an interval. This concludes the sketch of the proof.
\end{proof}

\begin{remark}
Recall that for any two bounded sets $A, B\subset \mathbb{R}$ we have
\[
\udim(A+B) \le \udim(A\times B) \le \udim(A) + \udim(B),
\]
but no such inequality holds for the lower box dimension in general. In contrast, for sumsets $A+A$ we do have a bound
\[
\ldim_B(A+A)\le\ldim_B(A\times A)\le 2\ldim_B A,
\]
since we can cover the product by squares coming from a cover of the components approximating the lower box dimension. Besides these ``product'' bounds, there are ``Pl\"{u}nnecke'' bounds between the different $\b_\ell$; see Proposition \ref{prop:positiveresult} in Section \ref{sec:plunnecke}. Thus finding the most general possible relations between the sequences $\a_n$, $\b_n$ and $\g_n$ appears to be rather difficult.

\end{remark}


\section{Pl\"{u}nnecke estimates for box-counting dimensions} \label{sec:plunnecke}

We begin by observing that in Theorem \ref{main}, if $\g_2=\g_1$, then necessarily $\g_\ell=\g_1$ for all $\ell$, so this theorem does not directly negate the possibility of a ``Pl\"{u}nnecke estimate'' for the upper box dimensions. However, it is possible to modify the construction to obtain counterexamples. We indicate how to show that for any $0<\gamma<1$ there exists a compact set $A\subset [0,1]$ such that
\[
\udim_B A = \udim_B(A+A) = \g,
\]
but
\[
\udim_B(A+A+A) = \min(1,3\g/2) > \g.
\]
Recall from Example \ref{alldim3} that there exist compact sets $A', A''\subset [0,1]$ such that
\[
\udim_B(A') = \g/2,\quad \udim_B(A'+A') = \g,\quad \udim_B(A'+A'+A') = \min(1,3\g/2),
\]
\[
\udim_B(A'') = \g,\quad \udim_B(A''+A'') = \g,\quad \udim_B(A''+A''+A'') = \g.
\]
Moreover, these sets are constructed by specifying types of blocks for finite sequences of binary digits $x_{n_k}\cdots x_{n_{k+1}-1}$, where $\{n_k\}$ is a rapidly increasing sequence. Now let $M_r$ be another rapidly increasing sequence, say $M_r = 2^{2^r}$. The set $A$ is defined by using the blocks corresponding to $A'$ for all $k\in [ M_{2 r -1}, M_{2r})$ for some $r$, and the blocks corresponding to $A''$ for $k\in [M_{2r}, M_{2r+1})$ for some $r$. It is then easy to check that
\[
\udim_B(i A) = \max(\udim_B(i A'),\udim_B(i A'')) \quad i=1,2,3,
\]
which yields the claim.

We finish the paper with the positive result mentioned in the introduction, which bounds the lower box dimension of iterated sumsets $\ell B$ in terms of the lower box dimensions of $A$ and $A+B$. The proof is a straightforward discretization argument using the Pl\"{u}nnecke-Rusza inequalities.

\begin{prop} \label{prop:positiveresult}
Let $A, B\subset \RR$ be bounded sets. Then for all $\ell\ge 2$,
\[
\ldim_B(\ell B) \le \ell\ldim_B(A+B) - (\ell-1)\ldim_B A.
\]
In particular, if $\ldim_B(A+A)=\ldim_B A$, then
\[
\ldim_B(\ell A) = \ldim_B A
\]
for all $\ell\in\NN$.
\end{prop}
\begin{proof}
Let $\D_{j,1}$ be the family $\{ [i 2^{-j}, (i+1) 2^{-j}]:i\in \ZZ\}$ of dyadic intervals of length $2^{-j}$, and for $\ell \ge 2$ let
\[
\D_{j,\ell}  = \{ [ i 2^{-j}, (i+\ell) 2^{-j} ] :
i\in \ZZ \}.
\]
Given a set $A\subset\RR$, write $\D_{j,\ell}(A)$ for the intervals in $\D_{j,\ell}$ having nonempty intersection with $A$. Note that for any sets $A_1,\ldots, A_\ell\subset\RR$ and any $j\ge 1$, if $b \in A_1+\ldots+A_\ell$, then there are $\I_i \in \D_{j,1}(A_i)$ such that $b\in I_1+\ldots+I_\ell \in \D_{j,\ell}(A_1+\ldots+A_k)$. Since $b$ belongs to at most $\ell+1$ elements of $\D_{j,\ell}$, we have
\begin{equation} \label{eq:dyadicsumsetbound}
|\D_{j,\ell}(A_1+\ldots+A_\ell)| \le (\ell+1)|\D_{j,1}(A_1) + \ldots + \D_{j,1}(A_\ell)|.
\end{equation}
Moreover, it is easy to see that for a fixed $\ell$ we have
\begin{equation} \label{eq:boxdimusinggendyadic}
\ldim_B(A) = \liminf_{j\rightarrow\infty}
\frac{\log|\D_{j,\ell}(A)|}{j}.
\end{equation}
(Recall that $\log$ is the base $2$ logarithm).

The Pl\"{u}nnecke-Rusza Theorem says that if $E, F$
are finite subsets of $\ZZ$ with $|E+F| < K |E|$, then $|\ell F| <
K^{\ell}|E|$; see \cite[Section 6.5]{TaoVu2006} for the proof and further background. We
apply this result to $E=\D_{j,1}(A), F=\D_{j,1}(B)$ and use
\eqref{eq:dyadicsumsetbound} to obtain
\begin{equation} \label{eq:applicationplunnecke}
|\D_{j,\ell}(B)| \le (\ell+1)\left(
\frac{|\D_{j,2}(A+B)|}{|\D_j(A)|} \right)^{\ell} |\D_j(A)|.
\end{equation}
Let $j_r\rightarrow\infty$ be a sequence such that
\[
\lim_{r\rightarrow\infty} \frac{\log|\D_{j_r,2}(A+B)|}{j_r} = \ldim_B(A+B).
\]
Using \eqref{eq:applicationplunnecke}, we conclude that
\begin{align*}
\ldim_B(\ell B) &\le \liminf_{r\rightarrow\infty}  \frac{\log |\D_{j_r,\ell}(B)|}{j_r} \\
&\le \ell \liminf_{r\rightarrow\infty} \frac{\log|\D_{j_r,2}(A+B)|}{j_r}  - (\ell-1) \limsup_{r\rightarrow\infty} \frac{\log|\D_{j_r}(A)|}{j_r}\\
&\le \ell \ldim_B(A+B) - (\ell-1) \ldim_B B.
\end{align*}
\end{proof}

\end{document}